\documentclass[reqno,10pt]{amsart}
\usepackage{amsmath,amsthm,amsfonts,amssymb,latexsym}
\usepackage{hyperref,times,enumitem}

\newtheorem{theorem}{Theorem}[section]
\newtheorem{lemma}{Lemma}[section]

\theoremstyle{remark}
\newtheorem{remark}{Remark}
\numberwithin{equation}{section}

\newcommand{\lb}{\left}
\newcommand{\rb}{\right}
\begin{document}
\title[On ``mixed'' modular equations of degree 21]
{On ``mixed'' modular equations of degree 21}
\author[S. Chandankumar]
{S. Chandankumar}
\address[S. Chandankumar]{Department of Mathematics, M. S. Ramaiah University of Applied Sciences, Peenya campus,
	Bengaluru - 560 058,
	Karnataka,India} \email{chandan.s17@gmail.comm }

\subjclass[2010]{33E05, 11F20}
\keywords{Modular equations, Theta--functions.}
\begin{abstract}
In the proposed work, we establish a total of six new $P$--$Q$ modular equations involving theta--function $f(-q)$ with moduli of orders 1, 3, 7 and 21. These equations can be regarded as modular identities in the alternate theory of signature 3. As a consequence, several values of quotients of theta--function are evaluated. 
\end{abstract}
\maketitle
\section{Introduction}
\noindent Throughout, we assume that $|q|<1$, Ramanujan's general theta--function $f(a,b)$ is defined by
\begin{equation}
f(a,b):=\sum_{n=-\infty}^{\infty}a^{n(n+1)/2}
b^{n(n-1)/2},\,\,\,|ab|<1.
\end{equation}
Furthermore following Ramanujan we define a case of $f(a,b)$:
\begin{equation}
f(-q):=f(-q,-q^2)=\sum_{n=-\infty}^{\infty}(-1)^nq^{n(3n-1)/2}.
\end{equation}
For  convenience throughout this paper, we set $f(-q^n)=f_n.$ The Gaussian hypergeometric function is defined by
$$_2F_1(a,b;c;z):=\sum_{n=0}^{\infty}\frac{\lb(a\rb)_n \lb(b\rb)_n}{\lb(c\rb)_n n!}z^n,\ \ \ 0\leq|z|<1,$$
where $a$, $b$, $c$ are complex numbers, $c\neq0,-1,-2,\ldots$, and $$(a)_0=1,\ \ (a)_n=a(a+1)\cdots(a+n-1)\ \ \textrm{for any positive integer}\ \ n.$$
Now we recall the notion of a ``mixed'' modular equation. Let $K(k)$ be the complete elliptic integral of the
first kind of modulus $k$. Recall that
\begin{equation}\label{ee11}
K(k):=\int_0^{\frac{\pi}{2}}\frac{d\phi}{\sqrt{1-k^2\sin^2\phi}}
=\frac{\pi}{2}\sum_{n=0}^{\infty}\frac{\left(\frac{1}{2}\right)^2_n}{\left(n!\right)^2}k^{2n},\,\,\,\,(0<k<1),
\end{equation}
and set $K'=K(k')$, where $k'=\sqrt{1-k^2}$  is called
complementary modulus of $k$. It is classical to set $q(k)=e^{-\pi
	K(k')/K(k)}$ so that $q$ is one-to-one and increases from 0 to 1.

Following Ramanujan we set $\alpha:=k^2$, and let $K$, $K'$, $L_1$, $L_1'$, $L_2$, $L_2'$, $L_3$ and $L_3'$ denote complete elliptic integrals of the first kind corresponding, in pairs, to the moduli $\sqrt{\alpha}$, $\sqrt{\beta}$, $\sqrt{\gamma}$ and $\sqrt{\delta}$, and their complementary moduli, respectively. Suppose that the equalities
\begin{equation}\label{ee14}
n_1\frac{K'}{K}=\frac{L_1'}{L_1},\ \
n_2\frac{K'}{K}=\frac{L_2'}{L_2} \ \ \textrm {and}\ \
n_3\frac{K'}{K}=\frac{L_3'}{L_3},
\end{equation}
hold for some positive integers $n_1$, $n_2$ and $n_3$. Then the relation between the moduli $\sqrt{\alpha}$, $\sqrt{\beta}$, $\sqrt{\gamma}$ and $\sqrt{\delta}$ that is induced by \eqref{ee14} is called as a ``mixed'' modular equation of composite degree $n_3=n_1n_2$.  We say that $\beta$, $\gamma$ and $\delta$ are of degrees $n_1$, $n_2$ and $n_3$ respectively over $\alpha$.  The multipliers $m=K/L_1$ and $m'=L_2/L_3$ are algebraic relations involving $\alpha$, $\beta$, $\gamma$ and $\delta$.

Ramanujan \cite{SR2} recorded total of seven $P$--$Q$ modular equations involving theta--function $f(-q)$ with moduli of orders 1, 3, 5 and 15. For example, he proved that 
\vspace{0.3cm}
\newline
If $P:=\dfrac{f(-q^3)f(-q^5)}{q^{1/3}f(-q)f(-q^{15})}$\ \ and \
\ $Q:=\dfrac{f(-q^6)f(-q^{10})}{q^{2/3}f(-q^2)f(-q^{30})}$, then

\begin{equation}\label{S26}
PQ+\frac{1}{PQ}=\lb(\frac{Q}{P}\rb)^3+\lb(\frac{P}{Q}\rb)^3+4.
\end{equation}

The proof of above equation by using classical methods can be found in \cite{BCB2}. In \cite{MSMCKSBH2} the authors have established new modular equations relating $P$ and $Q_r$, where
\begin{equation*}
P:=\frac{f^2(-q^3)}{q^{1/6}f(-q)f(-q^9)} \ \,\,\textrm{and}\,\,\ Q_{r}:=\frac{f^2(-q^{3r})}{q^{r/6}f(-q^r)f(-q^{9r})},
\end{equation*}
for {$r\in$ \{2, 3, 5, 7, 11, 13\} and using these modular relations they have explicitly evaluated several new cubic class invariants and cubic singular moduli. Also, recently in \cite{MSMCKSBH3} the authors have established new modular equations relating $A$ and $A_r$, where
	\begin{equation*}
	A:= \frac{f(-q) f(-q^{5})}{q^{1/2}f(-q^{3})f(-q^{15})}
	\,\,\,\,\ \textrm {and} \,\,\,\,\ A_{r}:= \frac{f(-q^r) f(-q^{5r})}{q^{r/2}f(-q^{3r})f(-q^{15r})},
	\end{equation*}
	for {$r\in$ \{3, 4, 5, 7, 9\} and explicitly evaluated few cubic singular moduli. For more details on $P$--$Q$ type modular equations one can refer \cite{MSMCKSHM1}, \cite{MSMCKSHM2} and \cite{MSMCKSMM}. 
	
	Motivated by all these works, we derive eight modular relations of composite degrees involving $f(-q)$ out of which six equations appears to be new to the literature. These new modular equations can also be identified as modular identities in the alternate theory of signature 3, which are also known as ``cubic modular equations''.

The paper is organized as follows. In Section \ref{Se2}, we collect some identities which are used in the subsequent sections. In Section \ref{Se3}, we establish several new $P$--$Q$ mixed modular equations which are akin to those recorded by Ramanujan in his notebooks. In section \ref{Se4}, we obtain various explicit evaluations of quotients of $f(-q),$ and $f(q)$.

\section{Preliminary results}\label{Se2}
In this section, we list some $P$--$Q$ modular equations given by Ramanujan which are helpful in proving our main results.
\begin{lemma}\cite[Ch. 25, Entry 68, p. 236]{BCB2}
	If $P:=\dfrac{f_1}{q^{1/4}f_7}$ \,\,\,\,\,and \,\,\,\,\,$Q:=\dfrac{f_3}{q^{3/4}f_{21}},$ \,\,\,\,then
	\begin{equation}\label{deg7}\\
		PQ+\frac{7}{PQ}=\left(\frac{Q}{P}\right)^2-3+\left(\frac{P}{Q}\right)^2.
	\end{equation}
\end{lemma}
\begin{lemma}\cite[Ch. 25, Entry 69, p. 236]{BCB2}
	If $P:=\dfrac{f_1}{q^{1/12}f_3}$ \,\,\,\,\,and \,\,\,\,\,$Q:=\dfrac{f_7}{q^{7/12}f_{21}},$ \,\,\,\,then
	\begin{equation}\label{deg3}\\
		(PQ)^3+\frac{27}{(PQ)^3}=\left(\frac{Q}{P}\right)^4-7\left(\frac{Q}{P}\right)^2+7\left(\frac{P}{Q}\right)^2-\left(\frac{P}{Q}\right)^4.
	\end{equation}
\end{lemma}
\begin{lemma}\cite[Ch. 25, Entry p. 204]{BCB2}
\text{If}\ \ $u:=\dfrac{f_1^2}{q^{1/6}f_3^2} $\ \ {and} \ \ $v:=\dfrac{f_2^2}{q^{1/3}f_6^2},$\ \ \  then
\begin{equation}\label{S21}
uv+\frac{9}{uv}=\lb(\frac{u}{v}\rb)^3+\lb(\frac{v}{u}\rb)^3.
\end{equation}
\end{lemma}
\begin{lemma}\cite[Ch. 20, Entry 1(iv), p. 345]{BCB1}
		\textrm{If}\ \ $u:=\dfrac{f_1^3}{q^{1/4}f_3^3}$ \ \ {and} \ \ $v:=\dfrac{f_3^3}{q^{3/4}f_9^3},$ \ \ {then}
	\begin{equation}\label{S22}
	uv+\frac{27}{uv}+9=\frac{v^2}{u^2}.
	\end{equation}
\end{lemma}
\begin{lemma}\cite[Ch. 25, Entry 62, p. 221]{BCB2}
	\textrm{If}\ \ $u:=\dfrac{f_1}{q^{1/12}f_3}$ \ \ {and} \ \ $v:=\dfrac{f_5}{q^{5/12}f_{15}},$ \ \ {then}
	\begin{equation}\label{S23}
	(uv)^2+5+\frac{9}{(uv)^2}=\frac{v^3}{u^3}-\frac{u^3}{v^3}.
	\end{equation}
\end{lemma}
 
\section{New ``mixed'' modular equations of degree 21}\label{Se3}
  We refresh our notation used in the previous section and let $u$, $w$, $r$ and $s$ be the modular functions defined by 
\begin{equation}\label{uv}
u:=\frac{qf_7f_{21}}{f_1f_3},\quad w:=\sqrt q\frac{f_1f_{21}}{f_3f_{7}},\quad r:=\frac{f_1}{q^{1/12}f_3} \quad\text{and} \quad s:=\frac{f_7}{q^{7/12}f_{21}}.
\end{equation}
\begin{lemma}The following identities holds:\label{lem0}
	\begin{equation}\label{R1}
	7u+\frac{1}{u}+3=w^2+\frac{1}{w^2}, 
	\end{equation}
	\begin{equation}\label{R2}
	r^6+\left(\frac{\sqrt{3}}{r}\right)^6=\left(1+7u\right)^3\times\sqrt{7u+1+\frac{1}{u}},
	\end{equation}
	and
	\begin{equation}\label{R3}
	s^6+\left(\frac{\sqrt{3}}{s}\right)^6=\left(1+\frac{1}{u}\right)^3\times\sqrt{7u+1+\frac{1}{u}}.
		\end{equation}
\end{lemma}
	Equation \eqref{R1} appear on p. 323 of Ramanujan's second notebook \cite{SR2}, proofs are given in \cite{BCB2}. Equations \eqref{R2} and \eqref{R3} were derived by using \eqref{deg7} and \eqref{deg3} by T. Anusha, et al. in \cite{ABCV}. Through out this section, we set	
\begin{equation}
u_{r}:= \dfrac{f_r f_{3r}}{q^rf_{7r}f_{21r}}
\,\,\,\,\ \textrm {and} \,\,\,\,\ w_{r}:= q^{r/2}\dfrac{f_rf_{21r}}{f_{3r}f_{7r}}.
\end{equation}

\begin{theorem}The following identity holds:
\begin{equation}\label{w2}
u_1^3-2u_1^2u_2-u_1u_2-7u_1^2u_2^2-2u_1u_2^2+u_2^3=0.
\end{equation}
\end{theorem}
\begin{proof}
Cubing equation \eqref{S21}, we obtain
\begin{equation}\label{S311}
d^{36}-729c^{12}d^{12}-c^{24}d^{24}+c^{36}-24c^{24}d^{12}-24c ^{12}d^{24}=0,
\end{equation}
\text{where}\,\, $c:=\dfrac{f_1^2}{q^{1/6}f_3^2}$ \  \text{and} \ $d:=\dfrac{f_2^2}{q^{1/3}f_6^2}$.

Invoking equation \eqref{R2}, we deduce that
\begin{equation}\label{S312}
c^6=\frac{M\pm\sqrt{M^2-108}}{2}\quad \text{and} \quad d^6=\frac{N\pm\sqrt{N^2-108}}{2}, 
\end{equation}
where \ 

$M:=(1+7u)^3\sqrt{7u+1+\dfrac{1}{u}},$ \quad  $N:=(1+7v)^3\sqrt{7v+1+\dfrac{1}{v}},\,\,\text{and} \,\,v=u(q^2).$

Using \eqref{S312} in \eqref{S311} and factorizing, we get
\begin{equation}\label{S313}\begin{split}
&M^6+N^6-324(M^4+N^4)+34992(M^2+N^2)+10287M^2N^2\\&+84(M^2N^4+M^4N^2)-M^4N^4=1259712.
\end{split}\end{equation}
On factorizing the above equations \eqref{S313} and using the definition of $M$ and $N$, we get  
\begin{equation}\label{w2224}
(u^3-2u^2v-uv-7u^2v^2-2uv^2+v^3)(1+47uvH(q))=0,
\end{equation}
where

\begin{equation*}
\begin{split}
&H(q)=b^{19}vu^2+(2u+7u^2+3)b^{19}-15a^{11}-2a^{12}-2118a-8261a^3-ua^{12}\\&-5346a^2-8926a^4-7242a^5-32146ua^5-2321a^7-948a^8-306a^9-78a^{10}\\&+(7a+189u^2+196u^3+ua^3+33)u^2b^{18}+(12691u^4+609u+14406u^5+250\\&+ua^5+55a^2+5537u^3)u^2b^{17}+(201ua^4+691ua^3+51ua^5+ua^7+235a^3+731a\\&+8ua^6+1446+539a^2)u^2b^{16}+(6726+1225a^4+2335a^3+4642a+69ua^7+ua^9\\&+321ua^6+10ua^8+1205ua^5+3485ua^4+4017a^2)u^2b^{15}+(6617ua^6+481ua^8\\&+17661a^3+1999ua^7+13a^{10}+24019a^2+25927+ua^{11}+11323a^4+2603a^6\\&+23574a+12ua^{10}+5891a^5)u^2b^{14}+(1019u^2a^{10}+19664u^2a^8+4830518u^3\\&+102223ua^4+12u^2a^{12}-1+5095u^2a^9+6024a^2+116903ua^3+44087ua^6\end{split}
\end{equation*}
\begin{equation*}
\begin{split}&+74728ua^5-2u+60247u^2a^7+21563ua^7+143u^2a^{11})vb^{12}+(237281u^2a^5\\&+2416421u^3-2+123062u^2a^6+123253u^3a^7+233974u^2+33281u^3a^8-6u\\&+1019u^3a^{10}+51699u^2a^7+a^{12}u^3+466099u^2a^3+372718u^2a^4+66960ua^3\\&+6910u^3a^9+13u^2u^{12})b^{12}+(80237u^2a^8+79511ua-15+155950u^3a^8-65u\\&+1459306u^2a^4+1560736u^2a^2+977828u^2a^5+7147588u^3+1710871u^2a^3\\&+527274u^2a^6+229007u^2a^7+33281u^3a^9+481u^3a^{11}+5095u^3a^{10}+10u^3a^{12})b^{11}\\&+(80237u^2a^9+1765010u^2a^6-464u+1999u^3a^{11}+3180208u^2a^5+162259ua\\&+4326688u^2a^2+4573523u^2a^4+281952u^2a^8-78+u^3a^{13}+u^25110750a^3\\&+19664u^3a^{10}+365308ua^2+123253u^3a^9+69u^3a^{12}+786256u^2a^7)b^{10}\\&+(4718884u^2a^6+229007u^2a^9+8306468u^2a^5+1424879ua^3+51699u^2a^{10}\\&+11605492u^2a^4+60247u^3a^{10}+6617u^3a^{11}+2147408u^2a^7-2439u-306\\&+321u^3a^{12}+1977801u^2+786256u^2a^8+254954a^5+110313a^3+8u^3a^{13})7b^9\\&
+(23894153u^2a^4+u^221563a^{11}+3563810ua^4+17521111u^2a^5+51u^3a^{13}\\&+u^24718884a^7+u^2527274a^9+10162311u^2a^6+1765010u^2a^8+387644a^4\\&
+u^3a^{14}+1364169ua^2+2603u^2a^{12}+59022a^2+u^2123062a^{10}+482u^3a^{12}\\&
-942-9990u)b^8+(29611502u^2a^5+201u^3a^{13}+17521111u^2a^6+977828u^2a^9\\&
+3180208u^2a^8+237281u^2a^{10}+5644543ua^5-2321-32824u+3485u^3a^{12}\\&
+8306468u^2a^7+580914ua+44087u^2a^{11}+1968153ua^2+5739395ua^4\\&
+6u^3a^{14}+5891u^2a^{12}+4196501ua^3)b^7+(4573523ua^9-597513u+691u^2a^{14}\\&
+23894153ua^7+372718ua^{11}+316825a^2+7367167a^6+11605492ua^8\\&
+175a^{15}+1459306ua^{10}+37u^2a^{15}+5644543a^7+74728ua^{12}+11323ua^{13}\\&
+7300698a^5-32146+2226056a^3+5130093a^4+u^2a^{16})vb^5+(5110750u^2a^8\\&+1777817ua^2
+4u^3a^{15}+7016029ua^4+110477ua+7300698ua^5+113u^3a^{14}\\&-175162u+10434u^2a^{10}
+2335u^2a^{13}+1710871u^2a^9+3563810ua^7+235u^2a^{14}\\&-7242+254954a^9+102223u^2a^{11}
+2523ua^{13}+5739395ua^6+4675710ua^3)v^5\\&+(4675710ua^5-322234au+5130093ua^5
+387644a^8+1424879ua^8+55u^2a^{15}\\&+680221ua^2+11a^{16}+27u^3a^{15}+4196501ua^6
+4017u^2a^{13}+4326688u^2a^8\\&+24019u^2a^{12}-8926+116903u^2a^{11}+66960ua^{11}
-39270a+1560736u^2a^9\\&+2782687ua^3+u^3a^{16})v^4+(4642u^2a^{13}+731u^2a^{14}
+87u^2a^{15}+23574u^2a^{12}\\&+345203u^2a^{10}-336223ua^2+98582u^2a^{11}+318008a^6
+365308ua^9-85283a^2\\&+2u^3a^{16}+ua^{17}+680221ua^3-8261+1364169ua^7
+1968153ua^6+1021084u^2a^9\\&-43653a+1777817ua^4+110313a^9)b^3+(59022a^8
+405536ua^5+12048ua^{12}
\end{split}
\end{equation*}
\begin{equation*}\label{w224}
\begin{split}
&+1446u^2a^{14}+3u^2a^{17}+162259ua^9+33a^{16}u^2
+233974u^2a^{10}+110477ua^4\\&+250u^2a^{15}+282543ua^8+79511ua^{10}+6726u^2a^{13}
-30980a-70974a^2\\&-85283a^3-5346+25927u^2a^{12}+u161778a^6+u^3a^{12}
-14669ua^3)b^2\\&-(175162ua^4+85359ua^5+2u^2a^{11}+32824ua^6+2439ua^8+2118+6ua^{11}\\&+65ua^{10}+9990ua^7+464ua^9+77450a^2+43653a^3
+39270a^4+12140a)b,
\end{split}
\end{equation*}
where \,\,$a=7u$\,\,and\,\,$b=7v$. We find that the first factor of \eqref{w2224} vanishes and second factor does not vanish for the sequence $\displaystyle \{q_n\}=\lb\{\frac{1}{1+n}\rb\}$.
Hence, first factor is identically equal to zero on $|q|<1$. By setting $u=u_1$ and $v=u_2$, we complete the proof.
\end{proof}
\begin{theorem}
\text{If}\ \ $P:=u_1 u_3$\ \ {and} \ \ $Q:=\dfrac{u_1}{u_3},$\ \ \  then
\begin{equation}\begin{split}\label{w3}
&Q^3+\frac{1}{Q^3}=5\left(Q^2+\frac{1}{Q^2}\right)+20\left(Q+\frac{1}{Q}\right)+\left(P+\frac{7^2}{P}\right)\left[Q+\frac{1}{Q}+1\right]\\&3\left(\sqrt{P}+\frac{7}{\sqrt P}\right)\left[\left(\sqrt {Q^3}+\frac{1}{\sqrt {Q^3}}\right)+3\left(\sqrt {Q}+\frac{1}{\sqrt {Q}}\right)\right]+42 . \end{split}\end{equation}
\end{theorem}
\begin{proof}
The proof of \eqref{w3} is similar to the proof of the equation \eqref{w2}, except that in place of the equation \eqref{S21} we use \eqref{S22}.
\end{proof}
\begin{theorem}
\text{If}\ \ $P:=u_1 u_5$\ \ {and} \ \ $Q:=\dfrac{u_1}{u_5},$\ \ \  then
\begin{equation}\begin{split}\label{w5}
&Q^3+\frac{1}{Q^3}=30\left(Q^2+\frac{1}{Q^2}\right)+215\left(Q+\frac{1}{Q}\right)+\left(P^2+\frac{7^4}{P^2}\right)+270\\& +15\left(P+\frac{7^2}{P}\right)\left[\sqrt{Q}+\frac{1}{\sqrt{Q}}\right]^2+5\left(\sqrt{P^3}+\frac{7^3}{\sqrt P^3}\right)\\&+5\left(\sqrt{P}+\frac{7}{\sqrt P}\right)\left[6\left(\sqrt {Q^3}+\frac{1}{\sqrt {Q^3}}\right)+19\left(\sqrt {Q}+\frac{1}{\sqrt {Q}}\right)\right]. \end{split}\end{equation}
\end{theorem}
\begin{proof}
	The proof of \eqref{w5} is similar to the proof of the equation \eqref{w2}, except that in place of the equation \eqref{S21} we use \eqref{S23}.
\end{proof}
\begin{theorem}
	The following identity holds:
	\begin{equation}\begin{split}\label{w7}
	&\frac{{u_1}^3}{{u_7}^3}=\frac{161{u_7}^2}{{u_1}}+\frac{112{u_7}^2}{{u_1}^2}+196\left({u_1}{u_7}+\frac{49}{{u_1}{u_7}}\right)+28\left({u_1}^2{u_7}^2+\frac{7^4}{{u_1}^2{u_7}^2}\right)+1372\\&
	+\frac{980{u_7}}{{u_1}}+\frac{1127{u_7}}{{u_1}^2}+\left({u_1}^4{u_7}^2+\frac{7^6}{{u_1}^2{u_7}^4}\right)+539\left({u_7}+\frac{7}{{u_1}}\right)+343\left({u_1}+\frac{7}{{u_7}}\right)\\&
	+77\left({u_7}^2{u_1}+\frac{7^3}{{u_1}^2{u_7}}\right)+49\left({u_7}{u_1}^2+\frac{7^3}{{u_1}{u_7}^2}\right)+7\left({u_7}^2{u_1}^3+\frac{7^5}{{u_1}^2{u_7}^3}\right)\\&
	+49\left({u_1}^2+\frac{7^2}{{u_7}^2}\right)+140\left({u_7}^2+\frac{7^2}{{u_1}^2}\right)+7\left({u_1}^3{u_7}+\frac{7^4}{{u_1}{u_7}^3}\right)+\frac{343{u_1}}{{u_7}}.
	 \end{split}\end{equation}
\end{theorem}
\begin{proof}
Replacing $q$ to $q^7$ in \eqref{R3} and equating with \eqref{R2}, we arrive at \eqref{w7}.
\end{proof}
\begin{theorem}
	\text{If}\ \ $P:=w_1 w_2$\ \ {and} \ \ $Q:=\dfrac{w_1}{w_2},$\ \ \  then
	\begin{equation}\begin{split}\label{u2}
	&P+\frac{1}{P}=Q^3+\frac{1}{Q^3}+4\left(Q+\frac{1}{Q}\right). \end{split}\end{equation}
\end{theorem}
\begin{proof}
	To begin with let us set
	\begin{equation}\label{w21}
	A:=w_1^2+\frac{1}{w_1^2}\,\,\text{and}\,\,B:=w_2^2+\frac{1}{w_2^2}.
	\end{equation}
	Observe that from \eqref{R1}, we have
	\begin{equation}\label{w22}
	u_{1}=\frac{A-3+a}{14}\,\,\,\text{and}\,\,\,u_{2}=\frac{B-3+b}{14},
	\end{equation}
	where $a:=(A-3)^2-28$ and $b:=(B-3)^2-28.$
	
	Using \eqref{w22} in \eqref{w2} and upon factorizing, we get
	\begin{equation}\label{w23}\begin{split}
	&4aw_1^4w_2^8+aw_1^2w_2^6-aw_1^{10}w_2^8-3aw_1^6w_2^4-9w_1^4w_2^6-9w_1^6w_2^4+2aw_1^6w_2^2\\&
	-7aw_1^4w_2^6+4aw_1^4w_2^4-aw_1^2w_2^4+4w_1^4w_2^2+4w_1^2w_2^4-19w_1^4w_2^4-36w_1^6w_2^6\\&
	-9w_1^8w_2^6-9w_2^8w_1^6-19w_1^8w_2^8-19w_1^8w_2^4-19w_1^4w_2^8+4w_1^{10}w_2^8+4w_1^{10}w_2^4\\&
	+4w_1^2w_2^8+4w_1^8w_2^{10}+4w_1^8w_2^2+4w_1^4w_2^{10}-w_1^{10}w_2^{10}-9w_1^{10}w_2^6-w_1^{10}w_2^2\\&
	-9w_1^6w_2^{10}-9w_1^6w_2^2-w_1^2w_2^{10}-9w_1^2w_2^6+2w_2^6w_1^{12}+2w_1^6w_2^{12}+8aw_1^6w_2^6\\&
	+2aw_1^6w_2^{10}-7aw_1^8w_2^6+aw_1^{10}w_2^6-3aw_1^6w_2^8+4aw_1^8w_2^8+4aw_1^8w_2^4-w_1^2w_2^2\\&
	-aw_1^{10}w_2^4-aw_1^2w_2^8+2w_1^6+2w_2^6+(aw_1^6w_2^4-3w_1^4w_2^6-7w_1^6w_2^4+aw_1^4w_2^6\\&
	-aw_1^4w_2^4-w_1^4w_2^2+4w_1^4w_2^4+8w_1^6w_2^6-3w_1^8w_2^6-7w_2^8w_1^6+4w_1^8w_2^8+4w_1^8w_2^4\\&
	+4w_1^4w_2^8-w_1^8w_2^{10}10-w_1^8w_2^2-w_1^4w_2^{10}+2w_1^{10}w_2^6+w_1^6w_2^{10}+w_1^6w_2^2+2w_1^2w_2^6\\&
	-4aw_1^6w_2^6+aw_1^8w_2^6+aw_1^6w_2^8-aw_1^8w_2^8-aw_1^8w_2^4-aw_1^4w_2^8)b=0,
	\end{split}\end{equation}
	Eliminating $b$ from the above equation, we get 
	\begin{equation}\label{w24}
	\begin{split}
	&(w_1^6w_2^6+1-w_1^4w_2^2+4w_1^2w_2^2+4w_1^4w_2^4-w_1^2w_2^4)(w_1^6+4w_1^4w_2^2-w_1^4w_2^4\\&
	+4w_1^2w_2^4-w_1^2w_2^2+w_2^6)(1+9w_2^4+18w_2^6-9w_2^2+9w_2^8-9w_2^{10}+w_2^{12}\\&+aw_2^2-6aw_2^8-6aw_2^4+aw_2^{10}+4aw_2^6)=0
	\end{split}
	\end{equation}
	We find that the second factor of \eqref{w24} vanishes and other factors does not vanish for the sequence $\displaystyle \{q_n\}=\lb\{\frac{1}{1+n}\rb\}$.
	Hence, second factor is identically equal to zero on $|q|<1$. By setting $P:=w_1 w_2$\ \ {and} \ \ $Q:=\dfrac{w_1}{w_2}$, we complete the proof.
\end{proof}
\begin{theorem}
\text{If}\ \ $P:=w_1 w_3$\ \ {and} \ \ $Q:=\dfrac{w_1}{w_3},$\ \ \  then
\begin{equation}\begin{split}\label{u3}
&P^2+\frac{1}{P^2}=Q^2+\frac{1}{Q^2}+\left(P+\frac{1}{P}\right)\left[3+\left(Q+\frac{1}{Q}\right)^2+3\left({Q}+\frac{1}{ Q}\right)\right]+1. \end{split}\end{equation}
\end{theorem}
\begin{proof}
	The proof of \eqref{u3} is similar to the proof of the equation \eqref{u2}, except that in place of the equation \eqref{w2} we use \eqref{w3}.
\end{proof}
\begin{theorem}
\text{If}\ \ $P:=F_1 F_5$\ \ {and} \ \ $Q:=\dfrac{F_1}{F_5},$\ \ \  then
\begin{equation}\begin{split}\label{u5}
&P^2+\frac{1}{P^2}=Q^3+\frac{1}{Q^3}+5\left(Q^2+\frac{1}{Q^2}\right)+20\left(Q+\frac{1}{Q}\right)\\&+5\left(P+\frac{1}{P}\right)\left[\left(Q+\frac{1}{Q}\right)+2\right]+35. 
\end{split}\end{equation}
\end{theorem}
\begin{proof}
	The proof of \eqref{u5} is similar to the proof of the equation \eqref{u2}, except that in place of the equation \eqref{w2} we use \eqref{w5}.
\end{proof}
\begin{theorem}
	\text{If}\ \ $P:=F_1 F_7$\ \ {and} \ \ $Q:=\dfrac{F_1}{F_7},$\ \ \  then
	\begin{equation}\begin{split}\label{u7} 
	&P^6+\frac{1}{P^6}+Q^4+\frac{1}{Q^4}=49\left(Q^2+\frac{1}{Q^2}\right)-28\left(Q+\frac{1}{Q}\right)-464\\&+\left(P^5+\frac{1}{P^5}\right)\left[29+7\left(Q+\frac{1}{Q}\right)\right]
	-\left(P^4+\frac{1}{P^4}\right)\left[106+14\left(Q+\frac{1}{Q}\right)\right]\\&
	+\left(P^3+\frac{1}{P^3}\right)\left[183+21\left(Q+\frac{1}{Q}\right)-7\left(Q^3+\frac{1}{Q^3}\right)+\left(Q^4+\frac{1}{Q^4}\right)\right]\\&
	-\left(P^2+\frac{1}{P^2}\right)\left[260+28\left(Q+\frac{1}{Q}\right)-49\left(Q^2+\frac{1}{Q^2}\right)+\left(Q^4+\frac{1}{Q^4}\right)\right]\\&
	+\left(P+\frac{1}{P}\right)\left[337+35\left(Q+\frac{1}{Q}\right)-49\left(Q^2+\frac{1}{Q^2}\right)+\left(Q^4+\frac{1}{Q^4}\right)\right].
	\end{split}\end{equation}
\end{theorem}
\begin{proof}
	The proof of \eqref{u7} is similar to the proof of the equation \eqref{u2}, except that in place of the equation \eqref{w2} we use \eqref{w7}.
\end{proof}
\begin{remark}
	For another proof of \eqref{w2} and \eqref{u2} one can look into \cite{BAM1}.
\end{remark}
\section{Explicit evaluations of ratios of $f(-q)$ and $f(q)$}\label{Se4}
In \cite{jy3}, J. Yi introduced two parameterizations $r_{k,n}$ and $r'_{k,n}$ as follows:
\begin{equation}\label{rkn}
r_{k,n}=\frac{f(-q)}{k^{1/4}q^{(k-1)/24}f({-q^k})},\,\,\,\textrm{where}\,\,\,q:=e^{-2 \pi\sqrt{n/k}}
\end{equation}
and
\begin{equation}\label{rkkn}
r'_{k,n}=\frac{f(q)}{k^{1/4}q^{(k-1)/24}f({q^k})},\,\,\,\textrm{where}\,\,\,q:=e^{-\pi\sqrt{n/k}}.
\end{equation}
and established several properties as well as explicit evaluations of $r_{k,n}$ and $r'_{k,n}$ for different positive rational values of $n$ and $k$.
In the following lemmas \ref{S27}--\ref{S29}, we list few properties of $r_{k,n}$ which will be used in sequel of this section to establish several new explicit evaluations of $r_{k,n}$ and $r'_{k,n}$.

\begin{lemma}\cite{jy3}\label{S27} For all positive real numbers $k$ and $n$,
\begin{enumerate}
\item[i)] $r_{k,1/n}\ r_{k,n}=1 \ \ and \ \ r'_{k,1/n}\ r'_{k,n}=1$,
\item [ii)] $r_{k,n}=r_{n,k} \ \ and \ \ r'_{k,n}=r'_{n,k}$.
\end{enumerate}
\end{lemma}

\begin{lemma}\cite{jy3}\label{S29}
For any positive real number $k$, $m$ and $n$,
\begin{enumerate}
\item[i)] $r_{k,n/m}={r_{mk,n}}/{r_{nk,m}}$,
\item[ii)] $r'_{k,n/m}={r'_{mk,n}}/{r'_{nk,m}}$.
\end{enumerate}
\end{lemma}

\begin{lemma}
	\begin{equation}\label{deg7r}
		\sqrt 7 \left(r_{7,n}r_{7,9n}+\frac{1}{r_{7,n}r_{7,9n}}\right)=
		\left(\frac{r_{7,9n}}{r_{7,n}}\right)^2-3+\left(\frac{r_{7,n}}{r_{7,9n}}\right)^2.
	\end{equation}
\end{lemma}
\begin{lemma}
	\begin{equation}\label{deg7rp}
		\sqrt 7 \left(r'_{7,n}r'_{7,9n}+\frac{1}{r'_{7,n}r'_{7,9n}}\right)=
		3-\left(\frac{r'_{7,9n}}{r'_{7,n}}\right)^2-\left(\frac{r'_{7,n}}{r'_{7,9n}}\right)^2.
	\end{equation}
	\begin{proof}
		Using \eqref{rkn} with $k=7$ in \eqref{deg7}, we arrive at \eqref{deg7r}. Replacing $q$ by $-q$ in \eqref{deg7}, then using \eqref{rkkn} in the resulting equation, we obtain \eqref{deg7rp}.
	\end{proof}
\end{lemma}
\begin{lemma}
	\begin{equation}\begin{split}\label{deg3r}
		&3\sqrt 3 \left((r_{3,n}r_{3,49n})^3+\frac{1}{(r_{3,n}r_{3,49n})^3}\right)=\left(\frac{r_{3,49n}}{r_{3,n}}\right)^4-7\left(\frac{r_{3,49n}}{r_{3,n}}\right)^2\\&+7\left(\frac{r_{3,n}}{r_{3,49n}}\right)^2-\left(\frac{r_{3,n}}{r_{3,49n}}\right)^4.
	\end{split}\end{equation}
\end{lemma}
\begin{lemma}
	\begin{equation}\label{deg3rr}\begin{split}
	&3\sqrt 3 \left((r'_{3,n}r'_{3,49n})^3+\frac{1}{(r'_{3,n}r'_{3,49n})^3}\right)=\left(\frac{r'_{3,49n}}{r'_{3,n}}\right)^4+7\left(\frac{r'_{3,49n}}{r'_{3,n}}\right)^2\\&-7\left(\frac{r'_{3,n}}{r'_{3,49n}}\right)^2-\left(\frac{r'_{3,n}}{r'_{3,49n}}\right)^4.
	\end{split}\end{equation}
\end{lemma}
\begin{proof}
	Using \eqref{rkn} with $k=3$ in \eqref{deg3}, we arrive at \eqref{deg3r}. Replacing $q$ by $-q$ in \eqref{deg3}, then using \eqref{rkkn} in the resulting equation, we obtain \eqref{deg7rp}.
\end{proof}
\begin{theorem}
	\begin{eqnarray}
		r_{7,6}&=&2^{-1/2}\left\{(\sqrt 7+\sqrt 3)(\sqrt 3+\sqrt 2)\right\}^{1/2}, \label{S1}\\
		r_{7,3/2}&=&2^{-1/2}\left\{(\sqrt 7-\sqrt 3)(\sqrt 3+\sqrt 2)\right\}^{1/2}, \label{S2}\\
		r'_{7,6}&=& 2^{-1/2}(5+\sqrt{21})^{1/2}(2-\sqrt{3})^{1/2}(\sqrt 7-\sqrt 6)^{1/4}(5+2\sqrt 6)^{1/4},\label{S3}\\
		r'_{7,3/2}&=& 2^{-1/2}(5-\sqrt{21})^{1/2}(2+\sqrt{3})^{1/2}(\sqrt 7-\sqrt 6)^{1/4}(5+2\sqrt 6)^{1/4}.\label{S4}
	\end{eqnarray}
\end{theorem}
\begin{proof}
	Employing \eqref{rkn} in \eqref{w2} with $n=1/6$ and recalling the fact that $r_{7,n}r_{7,1/n}=1$, we obtain
	\begin{equation}\label{p1}
		(r^4+\sqrt 7r^3+4r^2+\sqrt 7r+1)(r^2-\sqrt 7r+1)=0, \,\,\text{where}\,\,r=r_{7,6}r_{7,2/3}. 
	\end{equation}
	Since $r>1$, hence solving $r^2-\sqrt 7r+1=0,$ we get
	\begin{equation}\label{p2}
		r_{7,6}r_{7,2/3}=\frac{\sqrt 7+\sqrt 3}{2}. 
	\end{equation}
	Now set $n=1/6$ in \eqref{deg7r} and using \eqref{p2}, we deduce that
	\begin{equation}\label{p3}
		(x^2-2x\sqrt 3+1)(x^2+2x\sqrt 3+1)=0,  \,\,\text{where}\,\,x=r_{7,6}r_{7,3/2}. 
	\end{equation}
	Since $x>1$, hence solving $x^2-2x\sqrt 3+1=0,$ we find that
	\begin{equation}\label{p4}
		r_{7,6}r_{7,3/2}=\sqrt 3+\sqrt 2. 
	\end{equation}
	Using \eqref{p2} and \eqref{p4}, we obtain \eqref{S1} and \eqref{S2}. Now let us proceed to prove $r'_{7,6}$ and $r'_{7,3/2}.$ Replacing $q$ to $-q$ in \eqref{w2}, then employing \eqref{rkn} and \eqref{rkkn} in the resulting equation with $n=1/6$, we obtain 
	\begin{equation}\label{p5}
		R^3+2rR^2-\sqrt 7r^2R^2+\sqrt 7rR-2r^2R-r^3=0,
	\end{equation}  
	where $R:=r_{7,6}r_{7,2/3}$ and $r:=r'_{7,6}r'_{7,2/3}.$\\
	Using \eqref{p2} in \eqref{p5} and factoring we obtain
	\begin{equation}
		(2r+\sqrt 7-\sqrt 3)(10+2r+2\sqrt{21}+3\sqrt 7+5\sqrt 3)(10-2r+2\sqrt{21}-3\sqrt 7-5\sqrt 3)=0.
	\end{equation}
	Since $r>0$, we find that
	\begin{equation}\label{p6}
		r'_{7,6}r'_{7,2/3}=\frac{(5+\sqrt{21})(2-\sqrt{3})}{2}.
	\end{equation}
	Now set $n=1/6$ in \eqref{deg7rp} and using \eqref{p6}, we deduce that
	\begin{equation}\label{p7}
		x^4-(10\sqrt 7-24)x^2+1=0,\,\,\text{where}\,\,x=r'_{7,6}r'_{7,3/2}. 
	\end{equation}
	Since $x>1$, hence solving for $x,$ we obtain
	\begin{equation}\label{p8}
		r'_{7,6}r'_{7,3/2}=\sqrt{(\sqrt 7-\sqrt 6)(5+2\sqrt 6)}. 
	\end{equation}
	Using \eqref{p6} and \eqref{p8}, we obtain \eqref{S3} and \eqref{S4}, which completes the proof.
\end{proof}
\begin{theorem}
	\begin{eqnarray}
		&&r_{3,14}= (\sqrt 3+\sqrt 2)^{1/2}(2\sqrt 2+\sqrt 7)^{1/6}, \label{S5}\\
		&&r_{3,7/2}=(\sqrt 3+\sqrt 2)^{1/2}(2\sqrt 2-\sqrt 7)^{1/6}, \label{S6}\\
		&&r_{2,21}=2^{-1/2}(\sqrt 7+\sqrt 3)^{1/2}(2\sqrt 2+\sqrt 7)^{1/6}, \label{S7}\\
		&&r_{2,7/3}=2^{-1/2}(\sqrt 7+\sqrt 3)^{1/2}(2\sqrt 2-\sqrt 7)^{1/6}. \label{S8}
	\end{eqnarray}
\end{theorem}
\begin{proof}
	Employing \eqref{rkn} in \eqref{u2} with $n=1/14$ and recalling the fact that $r_{3,n}r_{3,1/n}=1$, we obtain
	\begin{equation}\label{p9}
		r^4-10r^2+1=0,\,\,\text{where}\,\,r=r_{3,14}r_{3,7/2}.
	\end{equation}
	Since $r>1$, hence solving \eqref{p9}, we get
	\begin{equation}\label{p10}
		r_{3,14}r_{3,7/2}=\sqrt{3}+\sqrt{2}.
	\end{equation}
	Now set $n=1/14$ in \eqref{deg3r} and using \eqref{p10}, we deduce that
	\begin{equation}\label{p11}
		x^6-4\sqrt 2x^3+1=0,  \,\,\text{where}\,\,x:=r_{7,14}r_{7,2/7}. 
	\end{equation}
	On solving the above equation for $x$ and since $x>0$, we get
	\begin{equation}\label{p12}
		r_{3,14}r_{3,2/7}=(\sqrt{7}+2\sqrt{2})^{1/3}.
	\end{equation}
	By \eqref{p10} and \eqref{p12}, we arrive at \eqref{S5} and \eqref{S6}. Now we proceed to prove \eqref{S7} and \eqref{S8}. 
	
	From Lemma \eqref{S29}, we have
	\begin{equation}\label{p23}
	r_{2,21}=r_{7,6}r_{3,2/7}\,\,\,\text{and}\,\,\,r_{2,7/3}=r_{7,6}r_{3,14}.
	\end{equation}
	Using \eqref{S1}, \eqref{S5}, and \eqref{S6}, we arrive at \eqref{S7} and \eqref{S8} respectively. Hence the proof.
\end{proof}
\begin{theorem}
	\begin{eqnarray}
		&&r_{3,21}=20^{-1/2}\left\{\sqrt{358+78\sqrt{21}}+\sqrt{342+78\sqrt{21}}\right\}^{1/4}\times a, \label{S9}\\
		&&r_{3,3/7}=20^{-1/2}\left\{\sqrt{358+78\sqrt{21}}-\sqrt{342+78\sqrt{21}}\right\}^{1/4}\times a, \label{S10}\\
		&&r'_{3,21}=2^{-11/24}3^{1/8}(\sqrt{7}+\sqrt{3})^{1/12}\left\{\sqrt{11+\sqrt{21}}+\sqrt{3+\sqrt{21}}\right\}^{1/4}, \label{S11}\\
		&&r'_{3,3/7}=2^{-11/24}3^{1/8}(\sqrt{7}+\sqrt{3})^{1/12}\left\{\sqrt{11+\sqrt{21}}-\sqrt{3+\sqrt{21}}\right\}^{1/4}, \label{S12}\\
		&&r_{7,9}=\frac{\sqrt{34+6\sqrt{21}}+\sqrt{18+6\sqrt{21}}}{4},\label{s79}\\
		&&r'_{7,9}=2^{-3/8}\left\{\sqrt{11+\sqrt{21}}+\sqrt{3+\sqrt{21}}\right\}^{1/4},\label{s79d}
	\end{eqnarray}
	where $a=(100\sqrt{7}+200\sqrt{3})^{1/6}(342+78\sqrt{21})^{1/12}$.
\end{theorem}
\begin{proof}
	Employing \eqref{rkn} in \eqref{u3} with $n=1/21$ and recalling the fact that $r_{3,n}r_{3,1/n}=1$, we obtain
	\begin{equation}\label{p13}
		h^2-13h-5=0,\,\,\text{where}\,\,h=r^2+\frac{1}{r^2},\,\,\,r=r_{3,21}r_{3,7/3}.
	\end{equation}
	Since $r>1$, hence solving \eqref{p13}, we get
	\begin{equation}\label{p14}
		r^2+\frac{1}{r^2}=\frac{13+3\sqrt{21}}{2}.
	\end{equation}
	On solving the above quadratic equation for $r$, we get \begin{equation}\label{p15}
		r_{3,21}r_{3,7/3}=\frac{\left(\sqrt{358+78\sqrt{21}}+\sqrt{342+78\sqrt{21}}\right)^{1/2}}{2}.
	\end{equation}
	Now setting $n=1/21$ in \eqref{deg3r} and using \eqref{p15}, we deduce that
	\begin{equation}\label{p16}
		(180x^3-(27\sqrt{7}-41\sqrt{3})(342+78\sqrt{21})^{1/2})	(10x^3-(\sqrt{7}+2\sqrt{3})(342+78\sqrt{21})^{1/2})=0,   
	\end{equation}
where\,\,$x:=r_{3,21}r_{3,3/7}$.

	On solving the above equation for $x$ and since $x>1$, we get
	\begin{equation}\label{p17}
		r_{3,21}r_{3,3/7}=\frac{(100\sqrt{7}+200\sqrt{3})^{1/3}(342+78\sqrt{21})^{1/6}}{10}.
	\end{equation}
	By \eqref{p15} and \eqref{p17}, we arrive at \eqref{S9} and \eqref{S10}. 
	To prove \eqref{s79}, set $k=3$, $n=7$, and $m=3$, in Lemma \ref{S29}, we get 
	\begin{equation}\label{p18}
		r_{7,9}=r_{3,21}r_{3,7/3}.
	\end{equation}
	Using the fact that $r_{3,n}r_{3,1/n}=1$, along with equations \eqref{S9} and \eqref{S10}, we arrive at \eqref{s79}. 
	
	Now we shall proceed to prove \eqref{S11} and \eqref{S12}. Change $q$ to $-q$ in \eqref{u3} then employing \eqref{rkkn} in the resulting equation with $n=1/21$, we get
	\begin{eqnarray}\label{p19}
	r^8-r^6-3r^4-r^2+1=0,\,\,\,\text{where}\,\,r:=r'_{3,21}r'_{3,7/3}.
	\end{eqnarray}
	Since $r>1$, we have on solving the above equation, we arrive at 
	\begin{eqnarray}\label{p20}
	r^2+\frac{1}{r^2}=\frac{\sqrt{21}+1}{2}.
	\end{eqnarray} 
	On solving the above equation \eqref{p20} and noting that $r>1$, we find that 
	\begin{equation}\label{p21}
	r'_{3,21}r'_{3,7/3}=2^{-3/4}\left\{\sqrt{11+\sqrt{21}}+\sqrt{3+\sqrt{21}}\right\}^{1/2}.	
	\end{equation}
	Now setting $n=1/21$ in \eqref{deg3rr} and using \eqref{p21}, we deduce that
	\begin{equation}\label{p22}
	\left(12h^3 -(\sqrt{7}-\sqrt{3})\left(\sqrt{6+2\sqrt{21}}\right)\right)\left(2h^3-\sqrt{54+18\sqrt{21}}\right)=0,	
	\end{equation}
	where \,\,$h:=r'_{3,21}r'_{3,3/7}$. Since $h>1$, we have
	\begin{equation}\label{p29}
	r'_{3,21}r'_{3,3/7}=2^{-1/3}\left(18+6\sqrt{21}\right)^{1/6}.
	\end{equation}
	Using \eqref{p21} and \eqref{p29}, we arrive at \eqref{S11} and \eqref{S12}.
	
	To prove \eqref{s79d}, set $k=3$, $n=7$, and $m=3$, in Lemma \ref{S29}, we get 
	\begin{equation}\label{p30}
	r'_{7,9}=r'_{3,21}r'_{3,7/3}.
	\end{equation}
	Using the fact that $r'_{3,n}r'_{3,1/n}=1$, along with equations \eqref{S11} and \eqref{S12}, we arrive at \eqref{s79d}. 
\end{proof} 
\begin{theorem}
	\begin{eqnarray}
	&&r_{7,21}= \left\{\frac{7\times 2^{1/3} 7^{1/6}+5\sqrt{7}+2\times2^{2/3}7^{5/6}}{3}\right\}^{1/2}\left\{\frac{\sqrt{147+a}+\sqrt{111+a}}{6}\right\}^{1/2}, \label{S13}\\
	&&r_{7,7/3}=\left\{\frac{7\times 2^{1/3} 7^{1/6}+5\sqrt{7}+2\times2^{2/3}7^{5/6}}{3}\right\}^{1/2}\left\{\frac{\sqrt{147+a}-\sqrt{111+a}}{6}\right\}^{1/2}, \label{S14}\\
	&&r_{3,49}=\frac{\sqrt{147+a}+\sqrt{111+a}}{6}, \label{S15}
	\end{eqnarray}
	where\,\,$a=21\times7^{2/3}2^{1/3}+37\times7^{1/3}2^{2/3}.$
\end{theorem}
\begin{proof}
	Employing \eqref{rkn} in \eqref{w7} with $n=1/21$ and recalling the fact that $r_{7,n}r_{7,1/n}=1$, we obtain
	\begin{equation}\label{p24}
	r^3-5\sqrt{7}r^2-7r-7\sqrt{7}=0, \,\,\text{where}\,\,r:=r_{7,21}r_{7,7/3}. 
	\end{equation}
	Since $r>1$, hence solving the above cubic equation, we get
	\begin{equation}\label{p25}
	r_{7,21}r_{7,7/3}=\frac{7\times 2^{1/3} 7^{1/6}+5\sqrt{7}+2\times2^{2/3}7^{5/6}}{3}. 
	\end{equation}
	Now set $n=1/21$ in \eqref{deg7r} and using \eqref{p25}, we deduce that
	\begin{equation}\label{p26}\begin{split}
	&(42x^2+(14\times7^{5/6}4^{1/3}+35\times7^{1/6}4^{2/3}+14\sqrt{7}+42)x\sqrt 3+1)\\&(42x^2-(14\times7^{5/6}4^{1/3}+35\times7^{1/6}4^{2/3}+14\sqrt{7}-42)x\sqrt 3+1)=0, 
	\end{split}
	\end{equation}
    where\,\,$x=r_{7,21}r_{7,3/7}$. 
	Since $x>1$, hence solving the second factor of the equation \eqref{p26}, we find that
	\begin{equation}\label{p27}
	r_{7,21}r_{7,3/7}=\frac{\sqrt{147+a}+\sqrt{111+a}}{6}, \text{where}\,\,a=21\times7^{2/3}2^{1/3}+37\times7^{1/3}2^{2/3}.
	\end{equation}
	Using \eqref{p25} and \eqref{p27}, we obtain \eqref{S13} and \eqref{S14}. 
	
	To prove \eqref{S15}, set $m=7,$ $n=3$, $k=7$ in Lemma \ref{S27}, we get 
	\begin{equation}\label{p28}
	r_{3,49}=r_{7,7/3}r_{7,1/21}.
	\end{equation}  
	Using \eqref{p28}, \eqref{S13} and \eqref{S14}, we arrive at \eqref{S15}. Hence the proof.
\end{proof}
%
%
%

\end{document}